\documentclass[11pt,reqno,a4paper]{amsart}

\usepackage[T1]{fontenc}
\usepackage[utf8]{inputenc}
\usepackage{lmodern}
\usepackage[margin=25mm]{geometry}

\usepackage{mathtools}
\usepackage{amssymb} 
\usepackage{enumitem}
\setlist[enumerate,1]{label=\textup{(\roman*)},ref=\roman*}

\usepackage[hidelinks]{hyperref}

\theoremstyle{plain}
\newtheorem{theorem}{Theorem}
\newtheorem{lemma}[theorem]{Lemma}

\newtheorem{corollary}[theorem]{Corollary}

\theoremstyle{definition}

\theoremstyle{remark}

\numberwithin{equation}{section}

\DeclareMathOperator{\Id}{Id}

\title[A Vector-Valued Extension of a Theorem of Marciszewski]{A Vector-Valued Extension of a Theorem of Marciszewski}
\author[L. Candido]{Leandro Candido}
\address{Universidade Federal de S\~ao Paulo -- UNIFESP,
Instituto de Ci\^encia e Tecnologia, Departamento de Matem\'atica,
S\~ao Jos\'e dos Campos, SP, Brazil}
\email{leandro.candido@unifesp.br}

\date{} 



\begin{document}

\begin{abstract}
We establish a vector-valued extension of a theorem of Marciszewski.
Let $K$ be a compact Hausdorff space, let $\Gamma$ be a set, and let
$X$ be a nonzero Banach space containing no isomorphic copy of $c_0$.
We prove that if $C(K)$ embeds isomorphically into $c_0(\Gamma,X)$,
then $C(K)$ is isomorphic to $c_0(\Lambda)$ for some set $\Lambda$.
Moreover, when $\Gamma$ is infinite, the optimal distortion of
embeddings of $C(K)$ into $c_0(\Gamma,X)$ equals that of embeddings
into $c_0(\Gamma)$. The classification result answers a question
posed in an earlier paper.
\end{abstract}
\maketitle

\section{Introduction}\label{sec:introduction}

Earlier work
\cite{CandidoGalegoHowFar,CandidoGalego,CandidoDistortion}
established lower bounds, in terms of the height of $K$,
for the distortion of isomorphic embeddings
\begin{equation}\label{Rel:Main}
C(K)\hookrightarrow c_0(\Gamma,X),
\end{equation}
where $K$ is a compact Hausdorff space, $\Gamma$ is a set,
and $X$ is a Banach space containing no isomorphic copy
of $c_0$. These bounds apply to both the scalar-valued
and the vector-valued settings.

In the separable case, these results yield a complete
classification. Indeed, if $C(K)$ is separable and
\eqref{Rel:Main} holds, then
\cite[Theorem~1.1]{CandidoDistortion} implies that $K$ is
scattered of finite height. If $K$ is infinite, the
Mazurkiewicz--Sierpi\'nski classification theorem
\cite[Theorem~8.6.10]{Se}, together with the theorem of
Bessaga and Pe{\l}czy\'nski \cite{BP}, yields
$C(K)\sim c_0$. If $K$ is finite, then
$C(K)$ is isometric to $c_0(|K|)$.
Thus, in the separable setting, every space $C(K)$
admitting \eqref{Rel:Main} is isomorphic to a scalar-valued
$c_0$ space.

The assumption that $X$ contains no isomorphic copy of $c_0$
is essential, since $C([0,1])$ embeds isometrically into
$c_0(\mathbb N,C([0,1]))$ but is not isomorphic to
$c_0(\Lambda)$ for any set $\Lambda$.

In the scalar-valued setting, Marciszewski
\cite[Theorem~1.1]{Marciszewski} proved that, for a compact
Hausdorff space $K$, the following assertions are equivalent:
\begin{enumerate}
\item $C(K)$ is isomorphic to $c_0(\Lambda)$ for some set
      $\Lambda$;
\item $C(K)$ is isomorphic to a subspace of $c_0(\Gamma)$ for some
      set $\Gamma$;
\item $C_p(K)$ is linearly homeomorphic to $C_p(A(\kappa))$ for
      some cardinal $\kappa$.
\end{enumerate}
Here $C_p(K)$ denotes $C(K)$ endowed with the topology of pointwise
convergence, and $A(\kappa)$ is the one-point compactification of a
discrete space of cardinality $\kappa$.

Motivated by the separable case, we asked in
\cite[Problem~1.3]{CandidoDistortion} whether Marciszewski's
characterization extends to the vector-valued setting:
if $X$ contains no isomorphic copy of $c_0$, does
\eqref{Rel:Main} imply that $C(K)\sim c_0(\Lambda)$
for some set $\Lambda$?

Related questions concerning embedding distortion and
Banach--Mazur distances have been investigated in
\cite{RondosBoundaries,RondosSomaglia,RondosCardinal,JaponMalecPiasecki}.
The classification question above has, to the best of our knowledge,
remained open.

In this note, we answer this question affirmatively.
For every nonzero Banach space $X$ containing no isomorphic
copy of $c_0$, the condition
\[
C(K)\hookrightarrow c_0(\Gamma,X)
\quad\text{for some set }\Gamma
\]
is equivalent to the conditions in Marciszewski's theorem.
Moreover, whenever $\Gamma$ is infinite,
\[
d_{\mathrm{emb}}\bigl(C(K),c_0(\Gamma,X)\bigr)
=
d_{\mathrm{emb}}\bigl(C(K),c_0(\Gamma)\bigr),
\]
where $d_{\mathrm{emb}}$ denotes the optimal embedding
distortion, defined below.
Thus the vector-valued factor $X$ neither enlarges the class
of spaces $C(K)$ admitting such embeddings nor improves
their optimal distortion. The proof uses compactness of
the coordinate operators to reduce the problem to
scalar-valued embeddings.

\section{Preliminaries}\label{sec:preliminaries}

Throughout this note, all Banach spaces are real.
For a Banach space $E$, we denote its closed unit ball
by $B_E$ and its unit sphere by $S_E$.

We write $E\hookrightarrow F$ if there exists an isomorphic
embedding of $E$ into $F$, and $E\sim F$ if these spaces
are isomorphic. The distortion of an isomorphic embedding
$T:E\to F$ is $\|T\|\|T^{-1}\|$, where $T^{-1}$ is defined
on $T(E)$. We define the optimal embedding distortion by
\[
d_{\mathrm{emb}}(E,F)
=
\inf\left\{
\|T\|\|T^{-1}\|:
T:E\to F \text{ is an isomorphic embedding}
\right\},
\]
with the convention that $\inf\varnothing=+\infty$.

For a compact Hausdorff space $K$, we denote by $C(K)$
the Banach space of continuous real-valued functions on $K$,
endowed with the supremum norm.

For a set $\Gamma$ and a Banach space $X$, the space
$c_0(\Gamma,X)$ consists of all families
$(x_\gamma)_{\gamma\in\Gamma}$ such that
$\{\gamma\in\Gamma:\|x_\gamma\|\geq\varepsilon\}$ is finite
for every $\varepsilon>0$, endowed with the norm $\|(x_\gamma)_{\gamma\in\Gamma}\|
=\sup_{\gamma\in\Gamma}\|x_\gamma\|$. We write $c_0(\Gamma)$ when $X=\mathbb R$.

For a bounded operator $T:E\to c_0(\Gamma,X)$, we call
$T_\gamma:E\to X$, defined by
\[
T_\gamma x=(Tx)_\gamma,
\]
the $\gamma$-th coordinate operator.

For further notation and terminology, we refer to
\cite{CandidoDistortion}.

\section{Main results}\label{sec:main-results}

We first recall the following consequence of
Pe{\l}czy\'nski's theorem
\cite[Theorem~7.6]{Biorthog} (see also \cite{PelczynskiV}).

\begin{lemma}\label{Lem:CompactCoordinates}
Let $K$ be a scattered compact Hausdorff space, and let $X$ be a
Banach space containing no isomorphic copy of $c_0$.
Then every bounded operator $T:C(K)\to X$ is compact.
\end{lemma}

\begin{lemma}\label{Lem:Scalarization}
Let $E$ be a nonzero Banach space, let $X$ be a Banach space,
and let $T:E\to c_0(\Gamma,X)$ be an isomorphic embedding.
Suppose that every coordinate operator $T_\gamma:E\to X$
is compact. Then, for every $0<\eta<1$, there exist a set
$\Lambda$ and an isomorphic embedding $R:E\to c_0(\Lambda)$
such that
\[
\|R\|\|R^{-1}\|
\leq \frac{\|T\|\|T^{-1}\|}{1-\eta}.
\]
Moreover, if $\Gamma$ is infinite, then $\Lambda$ can be
chosen with $|\Lambda|=|\Gamma|$.
\end{lemma}
\begin{proof}
Fix $0<\eta<1$, and set
$a=1/\|T^{-1}\|$ and $\delta=\eta a/2$.
For each $\gamma\in\Gamma$, compactness of $T_\gamma$ yields a
finite $\delta$-net
$\{y_{\gamma,1},\ldots,y_{\gamma,m_\gamma}\}$
for $T_\gamma(B_E)$, with $m_\gamma\geq1$.
For each $1\leq j\leq m_\gamma$, if $y_{\gamma,j}\neq0$,
the Hahn--Banach theorem provides $x_{\gamma,j}^*\in S_{X^*}$
such that
\[
\bigl|x_{\gamma,j}^*(y_{\gamma,j})\bigr|
=
\|y_{\gamma,j}\|.
\]
If $y_{\gamma,j}=0$, take $x_{\gamma,j}^*=0$.

Given $x\in B_E$, choose $j$ such that
$\|T_\gamma x-y_{\gamma,j}\|<\delta$. Then
\[
\|T_\gamma x\|
\leq
\|y_{\gamma,j}\|+\delta
=
\bigl|x_{\gamma,j}^*(y_{\gamma,j})\bigr|+\delta
\leq
\bigl|x_{\gamma,j}^*(T_\gamma x)\bigr|+2\delta.
\]
Thus, for every $x\in E$,
\begin{equation}\label{Eq:FiniteNorming}
\|T_\gamma x\|
\leq
\max_{1\leq j\leq m_\gamma}
\bigl|x_{\gamma,j}^*(T_\gamma x)\bigr|
+
2\delta\|x\|.
\end{equation}

Set $\Lambda=\{(\gamma,j):\gamma\in\Gamma,\ 1\leq j\leq m_\gamma\}$, and define $R:E\to c_0(\Lambda)$ by the formula
\[
(Rx)(\gamma,j)=x_{\gamma,j}^*(T_\gamma x).
\]
Since $1\leq m_\gamma<\infty$ for every $\gamma\in\Gamma$, we have $|\Lambda|=|\Gamma|$ whenever $\Gamma$ is infinite. To check that $R$ is well defined, fix $x\in E$ and $\varepsilon>0$. Since $|(Rx)(\gamma,j)|\leq\|T_\gamma x\|$ and $Tx\in c_0(\Gamma,X)$, only finitely many coordinates $\gamma$ can satisfy $|(Rx)(\gamma,j)|\geq\varepsilon$ for some $j$. Each such coordinate contributes only finitely many pairs $(\gamma,j)$, so $Rx\in c_0(\Lambda)$. Moreover, $R$ is linear and $\|R\|\leq\|T\|$.

Taking the supremum over $\gamma$ in
\eqref{Eq:FiniteNorming}, we obtain, for every $x\in E$,
\[
a\|x\|
\leq
\|Tx\|
\leq
\|Rx\|+2\delta\|x\|
=
\|Rx\|+\eta a\|x\|.
\]
Hence $\|Rx\|\geq(1-\eta)a\|x\|$. Thus $R$ is an isomorphic embedding and
\[
\|R^{-1}\|
\leq
\frac{\|T^{-1}\|}{1-\eta}.
\]
Together with $\|R\|\leq\|T\|$, this proves the desired estimate.
\end{proof}

\begin{theorem}\label{Thm:VectorMarciszewski}
Let $K$ be a nonempty compact Hausdorff space, let $\Gamma$ be a set, and let
$X$ be a nonzero Banach space containing no isomorphic copy of $c_0$.
If $C(K)$ is isomorphic to a subspace of $c_0(\Gamma,X)$, then
there exists a set $\Lambda$ such that
\[
C(K)\sim c_0(\Lambda).
\]
Moreover, if $\Gamma$ is infinite, then
\[
d_{\mathrm{emb}}\bigl(C(K),c_0(\Gamma,X)\bigr)
=
d_{\mathrm{emb}}\bigl(C(K),c_0(\Gamma)\bigr).
\]
\end{theorem}
\begin{proof}
Let $T:C(K)\to c_0(\Gamma,X)$ be an isomorphic embedding.
By \cite[Theorem 1.5]{Candido2016}, $K$ is scattered.
Hence Lemma~\ref{Lem:CompactCoordinates} implies that
every coordinate operator $T_\gamma$ is compact.
Lemma~\ref{Lem:Scalarization} therefore yields an
isomorphic embedding of $C(K)$ into a scalar-valued
$c_0$ space. By
\cite[Theorem~1.1, $(iv)\Leftrightarrow(v)$]{Marciszewski},
we conclude that $C(K)\sim c_0(\Lambda)$ for some set $\Lambda$.

Assume now that $\Gamma$ is infinite. For every $0<\eta<1$,
Lemma~\ref{Lem:Scalarization} and reindexing yield an
isomorphic embedding $R_\eta:C(K)\to c_0(\Gamma)$ with
\[
\|R_\eta\|\|R_\eta^{-1}\|
\leq \frac{\|T\|\|T^{-1}\|}{1-\eta}.
\]
Letting $\eta\to0$ and taking the infimum over all
isomorphic embeddings $T:C(K)\to c_0(\Gamma,X)$, we obtain
\[
d_{\mathrm{emb}}\bigl(C(K),c_0(\Gamma)\bigr)
\leq
d_{\mathrm{emb}}\bigl(C(K),c_0(\Gamma,X)\bigr).
\]

Conversely, fix $x_0\in S_X$. The map $J:c_0(\Gamma)\to c_0(\Gamma,X)$, given by $J((t_\gamma)_\gamma)=(t_\gamma x_0)_\gamma$, is an isometric embedding. Composing any isomorphic embedding of $C(K)$ into $c_0(\Gamma)$ with $J$ preserves
its distortion, yielding the reverse inequality.
\end{proof}

\begin{corollary}
Let $K$ be an infinite compact Hausdorff space, let $\Gamma$ be a set,
and let $X$ be a nonzero Banach space containing no isomorphic copy
of $c_0$. Then
\[C(K,X)\sim c_0(\Gamma,X)\Leftrightarrow C(K)\sim c_0(\Gamma).\]
\end{corollary}
\begin{proof}
Suppose that $C(K,X)\sim c_0(\Gamma,X)$. Then
\[
C(K)\hookrightarrow c_0(\Gamma,X)
\quad\text{and}\quad
c_0(\Gamma)\hookrightarrow C(K,X).
\]
Since $K$ is infinite, the first embedding and
\cite[Theorem~1.3]{Candido2016} imply that $\Gamma$
is infinite and $|K|\leq|\Gamma|$. Applying the same
theorem to the second embedding gives
$|\Gamma|\leq|K|$. Thus $|K|=|\Gamma|$.

By Theorem~\ref{Thm:VectorMarciszewski},
$C(K)\sim c_0(\Lambda)$ for some infinite set $\Lambda$.
Applying \cite[Theorem~1.3]{Candido2016} in both directions
with scalar-valued functions yields
$|\Lambda|=|K|=|\Gamma|$. Hence $C(K)\sim c_0(\Gamma)$.

Conversely, tensoring an isomorphism
$C(K)\sim c_0(\Gamma)$ with the identity $\Id_X$
(see \cite[Proposition~3.2]{Ryan}) and using the
canonical identifications
\[
C(K)\widehat\otimes_\varepsilon X\cong C(K,X),
\qquad
c_0(\Gamma)\widehat\otimes_\varepsilon X
\cong c_0(\Gamma,X),
\]
we obtain $C(K,X)\sim c_0(\Gamma,X)$.
\end{proof}

\begin{corollary}
Let $K$ be a nonempty compact Hausdorff space,, let $\Gamma$ be a set, and let
$(X_\gamma)_{\gamma\in\Gamma}$ be a family of Banach spaces, each
containing no isomorphic copy of $c_0$.
If $C(K)$ is isomorphic to a subspace of
$\big(\bigoplus_{\gamma\in\Gamma}X_\gamma\big)_{c_0}$, then
$C(K)\sim c_0(\Lambda)$ for some set $\Lambda$.
\end{corollary}
\begin{proof}
Set $X=\big(\bigoplus_{\gamma\in\Gamma}X_\gamma\big)_{\ell_2}$.
Every separable subspace of $X$ is contained in a countable
subsum. Hence \cite{Burlando} implies that $X$ contains
no isomorphic copy of $c_0$. 

Writing $i_\gamma:X_\gamma\to X$ for the canonical
inclusions, the map
\[
(x_\gamma)_{\gamma\in\Gamma}
\longmapsto
(i_\gamma x_\gamma)_{\gamma\in\Gamma}
\]
defines an isometric embedding of
$\big(\bigoplus_{\gamma\in\Gamma}X_\gamma\big)_{c_0}$
into $c_0(\Gamma,X)$. Since $K$ is nonempty, the assumed embedding ensures
that $X\neq\{0\}$. The conclusion now follows from
Theorem~\ref{Thm:VectorMarciszewski}.
\end{proof}

\section*{Acknowledgements}

This research was supported by the Funda\c{c}\~ao de Amparo \`a Pesquisa do Estado de S\~ao Paulo (FAPESP), Grant No.~2023/12916-1.

\end{document}